\documentclass{amsart}
\usepackage{amssymb}
\usepackage{graphicx}
\usepackage[shortlabels]{enumitem}
\usepackage{multirow}
\usepackage{amsmath,color}
\usepackage{hyperref}
\usepackage{url}
\usepackage{longtable}
\usepackage{tikz}
\theoremstyle{plain}% default
\newtheorem{theorem}{Theorem}[section]

\newtheorem{proposition}[theorem]{Proposition}
\newtheorem{lemma}[theorem]{Lemma}
\newtheorem{corollary}[theorem]{Corollary}

\theoremstyle{definition}

\theoremstyle{remark}
\newtheorem{remark}[theorem]{Remark}

{% This is the begin code
\def\R{\mbox{$\mathbb{R}$}}

\newcommand{\frc}[1]{\operatorname{frac} \left ( #1 \right )}

\newcommand{\Z}{\mathbb{Z}}

\begin{document}
\title[On the clique number of a strongly regular graph]{On the clique number of a \\strongly regular graph}

\author[G. R. W. Greaves]{ Gary R. W. Greaves }
\author[L. H. Soicher]{ Leonard H. Soicher }
\thanks{The first author was supported by JSPS KAKENHI; 
grant number: 26$\cdot$03903}

%%%%%%%%%%%%%%%%%%%%%%%%%%%%%%%%%%%%%%%%%%%%%%%%%%%%%%%%%%%%%%%%%%

\address{Research Center for Pure and Applied Mathematics,
Graduate School of Information Sciences, 
Tohoku University, Sendai 980-8579, Japan}
\email{grwgrvs@gmail.com}

\address{School of Mathematical Sciences, Queen Mary University of London, Mile End Road, London E1 4NS, UK}
\email{L.H.Soicher@qmul.ac.uk}

\subjclass[2010]{}


\begin{abstract}
	We determine new upper bounds for the clique numbers of strongly regular graphs in terms of their parameters.
	These bounds improve on the Delsarte bound for infinitely many feasible parameter tuples for strongly regular graphs, including infinitely many parameter tuples that correspond to Paley graphs.
\end{abstract}

\maketitle

\section{Introduction}

% Let $\Gamma$ be a strongly regular graph with parameters $(v,k,\lambda,\mu)$.
The \textbf{clique number} $\omega(\Gamma)$ of a graph $\Gamma$ is defined to be the cardinality of a clique of maximum size in $\Gamma$.
For a $k$-regular strongly regular graph with smallest eigenvalue $s < 0$, Delsarte~\cite[Section 3.3.2]{Del:73} proved that $\omega(\Gamma) \leqslant \lfloor 1 -k/s \rfloor$; we refer to this bound as the \textbf{Delsarte bound}.
Therefore, since one can write $s$ in terms of the parameters of $\Gamma$, one can determine the Delsarte bound knowing only the parameters $(v,k,\lambda,\mu)$ of $\Gamma$.
In this paper we determine new upper bounds for the clique numbers of strongly regular graphs in terms of their parameters.
Our bounds improve on the Delsarte bound infinitely often.

Let $q = p^k$ be a power of a prime $p$ congruent to $1$ mod $4$.
A \textbf{Paley graph} has vertex set equal to the finite field $\mathbb F_q$, and two vertices $a$ and $b$ are adjacent if and only if $a-b$ is a nonzero square.
For a Paley graph $\Gamma$ on $q$ vertices with $k$ even, Blokhuis~\cite{Blok84} showed that $\omega(\Gamma) = \sqrt{q}$; this corresponds to equality in the Delsarte bound.
Bachoc et al.~\cite{Bach:Paley13} recently considered the case when $\Gamma$ is a Paley graph on $q$ vertices with $k$ odd and, for certain such $q$, showed that $\omega(\Gamma) \leqslant \lfloor \sqrt{q} - 1 \rfloor$.
This corresponds to an improvement to the Delsarte bound for these Paley graphs.

Here, working much more generally, given a strongly regular graph $\Gamma$ with parameters $(v,k,\lambda,\mu)$, we provide inequalities in terms of the parameters of $\Gamma$ that, when satisfied, guarantee that the clique number of $\Gamma$ is strictly less than the Delsarte bound.
We show that these inequalities are satisfied by infinitely many feasible parameters tuples for strongly regular graphs and, in particular, are satisfied by infinitely many parameter tuples that correspond to Paley graphs.
Our inequalities are obtained using what we call the ``clique adjacency bound'' (see Section~\ref{sec:evaluating_the_clique_adjacency_polynomial}), a bound defined by the second author~\cite{Soi:CAB15}.
We also show that the clique adjacency bound is always at most the Delsarte bound when applied to strongly regular graphs.

The paper is organised as follows.
In Section~\ref{sec:mainresults} we state our main results and in Section~\ref{sec:parameters_of_strongly_regular_graphs} we state some standard identities that we will use in our proofs.
Section~\ref{sec:evaluating_the_clique_adjacency_polynomial} contains the proofs of our main results. 
In Section~\ref{sec:a_limitation_of_the_bound} we examine the strength of the clique adjacency bound and in Section~\ref{sec:HoffmanDelsarteCAB} we provide an illustrative example comparing certain bounds for the clique number of an edge-regular graph that is not necessarily strongly regular.
Finally, we give an appendix in which we describe our symbolic computations.

\section{Definitions and main results}
\label{sec:mainresults}
A non-empty $k$-regular graph on $v$ vertices is called \textbf{edge-regular} if there exists a constant $\lambda$ such that every pair of adjacent vertices has precisely $\lambda$ common neighbours.
The triple $(v,k,\lambda)$ is called the \textbf{parameter tuple} of such a graph.
% An edge-regular graph with parameters $(v,k,\lambda)$ is called \textbf{strongly regular} if there exists a further constant $\mu$ such that every pair of non-adjacent vertices has $\mu$ common neighbours.
% The \textbf{parameters} of a strongly regular graph are given by the quadruple $(v,k,\lambda,\mu)$.
A \textbf{strongly regular graph} $\Gamma$ with \textbf{parameter tuple} $(v,k,\lambda,\mu)$ is defined to be a non-complete  edge-regular graph with parameter tuple $(v,k,\lambda)$ such that every pair of non-adjacent vertices has precisely $\mu$ common neighbours.
We refer to the elements of the parameter tuple as the \textbf{parameters} of $\Gamma$.
We call the parameter tuple of a strongly regular graph \textbf{feasible} if its elements satisfy certain nonnegativity and divisibility constraints given by Brouwer~\cite[VII.11.5]{CRCHAndbook2007}).
%
% A strongly regular graph $\Gamma$ is called \textbf{primitive} if it is connected and its complement is also connected, otherwise it is called \textbf{imprimitive}.
% A strongly regular graph is imprimitive if and only if it or its complement is isomorphic to the disjoint union of complete graphs of equal size (see \cite[Section 10.1]{God01}).

Let $\Gamma$ be a strongly regular graph with parameters $(v,k,\lambda,\mu)$.
It is well-known that $\Gamma$ has at most three distinct eigenvalues, and moreover, the eigenvalues can be written in terms of the parameters of $\Gamma$ (see \cite[Section 10.2]{God01}).
% We will always assume that $\Gamma$ is connected and non-complete, in which case $\Gamma$ has precisely three distinct eigenvalues.
In what follows we denote the eigenvalues of $\Gamma$ as $k > r \geqslant s$.

% Following Cameron and Van Lint~\cite{CvL}, if the parameters of $\Gamma$ satisfy $(v - 1)(\mu - \lambda) = 2k$ then we say $\Gamma$ (and its parameter tuple) is \textbf{type I}.
% Type-I strongly regular graphs are also known as conference graphs.
% On the other hand, if $(\mu-\lambda)^2+4(k-\mu)$ is the square of an integer $d$, where $d$ divides $(v - 1)(\mu - \lambda) - 2k$ with quotient congruent to $v-1 \mod 2$ then we say $\Gamma$ (and its parameter tuple) is \textbf{type II}.
% Every strongly regular graph is type I, type II, or both type I and type II.
Strongly regular graphs whose parameters satisfy $k = (v-1)/2$, $\lambda = (v-5)/4$, and $\mu = (v-1)/4$ are called \textbf{type I} or \textbf{conference graphs}.
Strongly regular graphs all of whose eigenvalues are integers are called \textbf{type II}.
Every strongly regular graph is either type I, type II, or both type I and type II (see Cameron and Van Lint~\cite[Chapter 2]{CvL}).

The \textbf{fractional part} of a real number $a \in \R$ is defined as $\frc{a} := a - \lfloor a \rfloor$.
We are now ready to state our main results.

\begin{theorem}\label{thm:beatDelsarteConf}
	Let $\Gamma$ be a type-I strongly regular graph with $v$ vertices.
	Suppose that 
	\[
		0 < \frc{\sqrt{v}/2} < 1/4 + ( \sqrt{v}-\sqrt{v+5/4} )/2.
	\]
	Then $\omega(\Gamma) \leqslant \lfloor \sqrt{v} - 1 \rfloor$.
\end{theorem}
\begin{proof}
	Follows from Theorem~\ref{thm:CliqueAdjPol} together with Corollary~\ref{cor:beatDelsarte} below.
\end{proof}

\begin{remark}
	For a prime $p$ satisfying $0 < \frc{\sqrt{p}/2} < 1/4 + ( \sqrt{p}-\sqrt{p+5/4} )/2$, we have that $\lfloor \sqrt{p} \rfloor = 2\lfloor \sqrt{p}/2 \rfloor$ is even.
	Furthermore, for $n := \lfloor \sqrt{p} \rfloor$, since $\lfloor \sqrt{p} \rfloor > \sqrt{p+5/4} - 1/2$, we have $n^2 + n - 1 > p$.
	Hence, if $\Gamma$ is a Paley graph on $p$ vertices then $\omega(\Gamma) \leqslant \lfloor \sqrt{p} - 1 \rfloor$ by Bachoc et al.~\cite[Theorem 2.1 (i)]{Bach:Paley13} (see \cite[Remark 2.5]{Bach:Paley13}).
	Therefore, for type-I strongly regular graphs with $p$ (a prime) vertices, satisfying $0 < \frc{\sqrt{p}/2} < 1/4 + ( \sqrt{p}-\sqrt{p+5/4} )/2$, Theorem~\ref{thm:beatDelsarteConf} is a generalisation of Bachoc et al.~\cite[Theorem 2.1 (i)]{Bach:Paley13}.
\end{remark}

\begin{remark}
	Let $g$ be a positive integer.
	Then $(1+4g,2g,g-1,g)$ is a feasible parameter tuple for a type-I strongly regular graph on $v = 1+4g$ vertices. 
	Observe that $( \sqrt{v}-\sqrt{v+5/4} )/2$ tends to $0$ as $v$ tends to infinity.
	Using Fej\'{e}r's theorem (see Kuipers and Niederreiter~\cite[page 13]{Kuip:1974uniform}), it is straightforward to show that the sequence $(\sqrt{1+4g}/2)_{g \in \mathbb{N}}$ is uniformly distributed modulo $1$.
	Therefore we can apply Theorem~\ref{thm:beatDelsarteConf} to about a quarter of all feasible parameter tuples for type-I strongly regular graphs.
	
	Let $\mathcal P$ denote the set of all primes $p$ of the form $p = 1+4g$ for some $g \in \mathbb N$.
	Then the sequence $(\sqrt{p}/2)_{p \in \mathcal{P}}$ is uniformly distributed modulo $1$ (see Balog~\cite[Theorem 1]{Balog:85}).
	Therefore, since Paley graphs on $p$ vertices exist for all $p \in \mathcal P$, Theorem~\ref{thm:beatDelsarteConf} is applicable to infinitely many strongly regular graphs.
\end{remark}

Note that the example in \cite{Soi:CAB15} with parameters $(65,32,15,16)$ is an example of a (potential) graph satisfying the hypothesis of Theorem~\ref{thm:beatDelsarteConf}.

A graph is called \textbf{co-connected} if its complement is connected.
We have the following:

\begin{theorem}\label{thm:beatDelsarteNonConf}
	Let $\Gamma$ be a co-connected type-II strongly regular graph with parameters $(v,k,\lambda,\mu)$ and eigenvalues $k > r \geqslant s$.
	Suppose that 
	\[
		0 < \frc{-k/s} < 1 - (r^2 + r)/(v - 2k + \lambda).
	\]
	Then $\omega(\Gamma) \leqslant \lfloor -k/s \rfloor$.
\end{theorem}
\begin{proof}
	Follows from Theorem~\ref{thm:CliqueAdjPol} together with Corollary~\ref{cor:beatDel01} below.
\end{proof}

\begin{remark}
	Currently Brouwer~\cite{Brouwer:URL} lists the feasible parameter tuples for connected and co-connected strongly regular graphs on up to $1300$ vertices.
	Of these, about $1/8$ of the parameter tuples of type-II strongly regular graphs satisfy the hypothesis of Theorem~\ref{thm:beatDelsarteNonConf}.
	By the remark following Corollary~\ref{cor:beatDel01}, it follows that Theorem~\ref{thm:beatDelsarteNonConf} can be applied to about $1/4$ of the complementary pairs of type-II strongly regular graphs on Brouwer's list.
\end{remark}

% \textbf{***Can we characterise an infinite family satisfying Theorem~\ref{thm:beatDelsarteNonConf}?***}
Note that the example in \cite{Soi:CAB15} of a strongly regular graph with parameter tuple $(144,39,6,12)$ is an example of a graph satisfying the hypothesis of Theorem~\ref{thm:beatDelsarteNonConf}; in fact, in this case, the conclusion of Theorem~\ref{thm:beatDelsarteNonConf} is satisfied with equality.
The parameter tuple $(88,27,6,9)$ is the first parameter tuple in Brouwer's list to which we can apply Theorem~\ref{thm:beatDelsarteNonConf} and whose corresponding graphs are not yet known to exist (or not exist).

\section{Parameters of strongly regular graphs} % (fold)
\label{sec:parameters_of_strongly_regular_graphs}

Here we state some well-known properties of strongly regular graphs and their parameters.
The first two propositions are standard (see Brouwer and Haemers~\cite[Chapter 9]{brou:spec11} or Cameron and Van Lint~\cite[Chapter 2]{CvL}).

\begin{proposition}\label{pro:SRG}
	Let $\Gamma$ be a strongly regular graph with parameters $(v,k,\lambda,\mu)$ and eigenvalues $k > r \geqslant s$.
	Then
	\begin{align*}
		(v-k-1)\mu &= k(k-\lambda-1); \\
		\lambda-\mu &= r+s; \\
		\mu - k &= rs.
	\end{align*}
\end{proposition}

\begin{proposition}\label{pro:SRGtI}
	Let $\Gamma$ be a type-I strongly regular graph with parameters $(v,k,\lambda,\mu)$ and eigenvalues $k > r > s$.
	Then
	\begin{align*}
		k = (v-1)/2; \quad \lambda = (v-&5)/4; \quad \mu = (v-1)/4; \\
		r = (\sqrt{v}-1)/2; \quad\quad\quad & \quad\quad s = -(\sqrt{v}+1)/2.
	\end{align*}
\end{proposition}

The next proposition is a key observation.

\begin{proposition}\label{pro:fracPart}
	Let $\Gamma$ be a strongly regular graph with parameters $(v,k,\lambda,\mu)$ and eigenvalues $k > r \geqslant s$.
	\begin{enumerate}[(i)]
		\item If $\Gamma$ is type I then $k/s - 2\frc{\mu/s}$ is an integer.
		\item If $\Gamma$ is type II then $k/s-\frc{\mu/s}$ is an integer.
	\end{enumerate}
\end{proposition}
\begin{proof}
	If $\Gamma$ is type I then, by Proposition~\ref{pro:SRGtI}, we have $k-2\mu = 0$.
	If $\Gamma$ is type II then, by Proposition~\ref{pro:SRG}, we have $k/s-\mu/s = -r$ and $r$ is an integer. 
\end{proof}

Next, the complement $\overline \Gamma$ of a strongly regular graph $\Gamma$ is also a strongly regular graph.
This is again a standard result (see Cameron and Van Lint~\cite[Chapter 2]{CvL}).

\begin{proposition}\label{pro:comp}
	Let $\Gamma$ be a connected and co-connected strongly regular graph with parameters $(v,k,\lambda,\mu)$ and eigenvalues $k > r > s$.
	Then $\overline \Gamma$ is strongly regular with parameters $(v,v-k-1,v-2k+\mu-2,v-2k+\lambda)$ and eigenvalues $v-k-1 > -s-1 > -r-1$.
\end{proposition}

Finally we state some straightforward bounds for the parameters of strongly regular graphs.

\begin{proposition}\label{pro:SRGbound}
	Let $\Gamma$ be a strongly regular graph with parameters $(v,k,\lambda,\mu)$.
	Then 
	\begin{enumerate}[(i)]
		\item $v - 2k + \lambda \geqslant 0$ with equality if and only if $\Gamma$ is complete multipartite;
		\item $k - \lambda - 1 \geqslant 0$ with equality if and only if $\overline \Gamma$ is complete multipartite.
	\end{enumerate}
\end{proposition}

% For the remainder of the paper we will distinguish the two types of strongly regular graph and we will write $(v,k,\lambda,\mu)$-type-I (resp. II) graph to denote a type-I (resp. II) strongly regular graph with parameter tuple $(v,k,\lambda,\mu)$.

% section parameters_of_strongly_regular_graphs (end)

\section{The clique adjacency polynomial} % (fold)
\label{sec:evaluating_the_clique_adjacency_polynomial}

Now we define our main tool, the clique adjacency polynomial.
Given an edge-regular graph $\Gamma$ with parameters $(v,k,\lambda)$, define the \textbf{clique adjacency polynomial} $C_\Gamma(x,y)$ as
\[
	C_\Gamma(x,y) := x(x+1)(v-y) - 2xy(k-y+1) + y(y-1)(\lambda-y+2).
\]

The utility of the clique adjacency polynomial follows from~\cite[Theorem 1.1]{Soi:10} (see also~\cite[Theorem 3.1]{Soi:CAB15}), giving:

\begin{theorem}\label{thm:CliqueAdjPol}
	Let $\Gamma$ be an edge-regular graph with parameters $(v,k,\lambda)$.
	Suppose that $\Gamma$ has a clique of size $c \geqslant 2$.
	Then $C_\Gamma(b,c) \geqslant 0$ for all integers $b$.
\end{theorem}

As discussed in \cite{Soi:10} and \cite{Soi:CAB15}, Theorem~\ref{thm:CliqueAdjPol} provides a way of bounding the clique number of an edge-regular graph using only its parameters.
Indeed, by Theorem~\ref{thm:CliqueAdjPol}, for an edge-regular graph $\Gamma$ and some integer $c \geqslant 2$, if there exists an integer $b$ such that $C_\Gamma(b,c) < 0$ then $\omega(\Gamma) \leqslant c-1$.
% When $c$ is the least integer greater than $1$ such that there exists $b \in \mathbb Z$ with $C_\Gamma(b,c)<0$, we refer to $c-1$ as the
Hence we define the \textbf{clique adjacency bound} (CAB) to be the least integer $c \geqslant 2$ such that $C_\Gamma(b,c+1) < 0$ for some $b \in \mathbb Z$; note that such a $c$ always exists.
% We refer to such a bound, obtained using Theorem~\ref{thm:CliqueAdjPol}, as the \textbf{clique adjacency bound}.

We will show that, for a $k$-regular strongly regular graph $\Gamma$, the clique adjacency bound gives $\omega(\Gamma) \leqslant \lfloor 1 -k/s \rfloor$.
That is, the clique adjacency bound is always at least as good as the Delsarte bound when applied to strongly regular graphs.
This follows from Theorem~\ref{thm:CliqueAdjPol} together with Theorem~\ref{thm:CAB-Del} below.
More interestingly, we will also show that the clique adjacency bound does better than the Delsarte bound for infinitely many feasible parameter tuples for strongly regular graphs.
In this section we consider the univariate polynomial $C_\Gamma(f(t), g(t))$ in the variable $t$, where $f(t)$ and $g(t)$ are linear polynomials in $t$.
The main idea is to choose the linear polynomials $f$ and $g$ such that there exists $t \in \R$ such that $C_\Gamma(f(t), g(t)) < 0$, $f(t) \in \Z$, and $g(t)$ is an integer at least $2$.
We begin by stating one of the main results of this paper. 

\begin{theorem}\label{thm:CAB-Del}
	Let $\Gamma$ be a strongly regular graph with parameters $(v,k,\lambda,\mu)$ and eigenvalues $k > r \geqslant s$.
	Then, 
	$$C_\Gamma \left ( \left \lfloor -\mu/s \right \rfloor, \left \lfloor 2-k/s \right \rfloor \right ) < 0.$$
\end{theorem}
\begin{proof}
	Follows from Corollary~\ref{cor:meetDelsarte} and Corollary~\ref{cor:meetDelsarte1} below.
\end{proof}

Observe that, together with Theorem~\ref{thm:CliqueAdjPol}, Theorem~\ref{thm:CAB-Del} shows that the clique adjacency bound always does as well as the Delsarte bound for strongly regular graphs.

Now we can state our first polynomial identity, which shows that the clique adjacency polynomial is negative at a certain point.

\begin{lemma}\label{lem:CAP0}
	Let $\Gamma$ be a connected strongly regular graph with parameters $(v,k,\lambda,\mu)$ and eigenvalues $k > r > s$.
	Then
	\[
		C_\Gamma(-\mu/s, 2-k/s) = {(2s-r)(r+1)} < 0.
	\]
\end{lemma}
\begin{proof}
	The equality follows from direct calculation (see Appendix~\ref{sec:comp}), using the equalities in Proposition~\ref{pro:SRG}.
	The right-hand side is negative since $s < 0$ and $r \geqslant 0$.
\end{proof}

Let $\Gamma$ be a strongly regular graph with parameters $(v,k,\lambda,\mu)$ such that both $\mu/s$ and $k/s$ are integers.
Then by Lemma~\ref{lem:CAP0}, together with Theorem~\ref{thm:CliqueAdjPol}, we recover the Delsarte bound, i.e., $\omega(\Gamma) \leqslant \left \lfloor 1-k/s \right \rfloor$.
It remains for us to deal with the situation when $k/s$ and $\mu/s$ are not integers.
In the remainder of this section, motivated by Lemma~\ref{lem:CAP0}, we consider integral points $(x,y) \in \Z^2$ close to $(-\mu/s, 2-k/s)$ such that $C_\Gamma(x,y)$ is negative. 
We deal with the type I and type II cases separately.

\subsection{Type-I strongly regular graphs} % (fold)
\label{sub:conference_graphs}

Let $\Gamma$ be a type-I strongly regular graph (or conference graph) with $v$ vertices.
By Proposition~\ref{pro:SRGtI} we have $-\mu/s = r$ and $-k/s = 2r$.
Therefore, we consider integral points $(x,y)$ close to $(r,2+2r)$ at which to evaluate the clique adjacency polynomial.
In view of Proposition~\ref{pro:fracPart}, we evaluate $C_\Gamma(x,y)$ at points of the form $(r-t,a+2r-2t)$ for some $a \in \mathbb N$, thinking of $t$ as the fractional part of $r$.

\begin{lemma}\label{lem:CAP2.0}
	Let $\Gamma$ be a type-I strongly regular graph with $v$ vertices and eigenvalues $k > r > s$.
	Then
	\begin{align}
		C_\Gamma(r - t, 3 + 2r - 2t) &= 2(t-1)(t+s-2)(t+2s); \label{eqn:confPol3} \\
		C_\Gamma(r - t, 2 + 2r - 2t) &= (t+s)(2t^2+(4s-1)t-3s-1). \label{eqn:confPol2}
	\end{align}
\end{lemma}
\begin{proof}
	The equalities follow from direct calculation (see Appendix~\ref{sec:comp}), applying Proposition~\ref{pro:SRG} and the definition of a type-I strongly regular graph.
\end{proof}
%
% \begin{lemma}\label{lem:CAP2.1}
% 	Let $\Gamma$ be a conference graph with $v$ vertices.
% 	Then
% 	\[
% 		C_\Gamma(r - t, 2 + 2r - 2t) = (t+s)(2t^2+(4s-1)t-3s-1).
% 	\]
% 	Hence, $C_\Gamma(r - t, 2 + 2r - 2t)$ is a positive cubic polynomial in $t$ with smallest zero $3/4 + ( \sqrt{v}-\sqrt{v+5/4} )/2$.
% \end{lemma}
% \begin{proof}
% 	The equality follows from direct calculation, using the equalities in Proposition~\ref{pro:SRG}.
% \end{proof}
% 
The right-hand side of Equation~\eqref{eqn:confPol2} is a cubic polynomial in the indeterminate $t$ with positive leading coefficient.
Furthermore, since for a type-I strongly regular graph we have $s = -(\sqrt{v}+1)/2$, we observe that the smallest zero of the right-hand side of Equation~\eqref{eqn:confPol2} is equal to $3/4 + ( \sqrt{v}-\sqrt{v+5/4} )/2$.
Hence $C_\Gamma(r - t, 2 + 2r - 2t)$ is negative for $t < 3/4 + ( \sqrt{v}-\sqrt{v+5/4} )/2$.
We use this observation in the next result, which can be used with Theorem~\ref{thm:CliqueAdjPol} to obtain the Delsarte bound for conference graphs.

\begin{corollary}\label{cor:meetDelsarte}
	Let $\Gamma$ be a type-I strongly regular graph with $v$ vertices.
	Then $$
	C_\Gamma \left ( \left \lfloor (\sqrt{v}-1)/2  \right \rfloor, \left \lfloor \sqrt{v} + 1 \right \rfloor \right ) < 0.
	$$
\end{corollary}
\begin{proof}
	Let $t = \frc{r}$.
	If $t > 1/2$ then 
	\[
		C_\Gamma \left ( \left \lfloor (\sqrt{v}-1)/2  \right \rfloor, \left \lfloor \sqrt{v} + 1 \right \rfloor \right ) = C_\Gamma(r - t, 3 + 2r - 2t)
	\]
	and the right-hand side of Equation~\eqref{eqn:confPol3} is negative for $t < 1$.
	Otherwise, if $t < 1/2$ then
	\[
		C_\Gamma \left ( \left \lfloor (\sqrt{v}-1)/2  \right \rfloor, \left \lfloor \sqrt{v} + 1 \right \rfloor \right ) = C_\Gamma(r - t, 2 + 2r - 2t),
	\]
	which is negative since $t < 1/2 < 3/4 + ( \sqrt{v}-\sqrt{v+5/4} )/2$.
	Note that $t$ cannot be equal to $1/2$ since $r$ is an algebraic integer.
\end{proof}

The next corollary follows in a similar fashion.

\begin{corollary}\label{cor:beatDelsarte}
	Let $\Gamma$ be a type-I strongly regular graph with $v$ vertices.
	Suppose that 
	\[
		0 < \frc{\sqrt{v}/2} < 1/4 + ( \sqrt{v}-\sqrt{v+5/4} )/2.
	\]
	Then $C_\Gamma \left ( \left \lfloor (\sqrt{v}-1)/2 \right \rfloor, \left \lfloor \sqrt{v} \right \rfloor \right) < 0$.
\end{corollary}
\begin{proof}
	Let $t = \frc{r} = \frc{(\sqrt{v}-1)/2}$.
	Then by our hypothesis $1/2 < t < 3/4 + ( \sqrt{v}-\sqrt{v+5/4} )/2$.
	Therefore we have
	\[
		C_\Gamma \left ( \left \lfloor (\sqrt{v}-1)/2  \right \rfloor, \left \lfloor \sqrt{v} \right \rfloor \right ) = C_\Gamma(r - t, 2 + 2r - 2t),
	\]
	which is negative since $t < 3/4 + ( \sqrt{v}-\sqrt{v+5/4} )/2$.
\end{proof}

% subsection conference_graphs (end)

\subsection{Type-II strongly regular graphs} % (fold)
\label{sub:non_conference_strongly_regular_graph}

Let $\Gamma$ be a type-II strongly regular graph with parameters $(v,k,\lambda,\mu)$.
Again, in view of Proposition~\ref{pro:fracPart}, we evaluate $C_\Gamma(x,y)$ at points of the form $(-\mu/s - t,a -k/s - t)$ for some $a \in \mathbb Z$, thinking of $t$ as the fractional part of $-\mu/s$.

\begin{lemma}\label{lem:CAP1}
	Let $\Gamma$ be a strongly regular graph with parameters $(v,k,\lambda,\mu)$ and eigenvalues $k > r \geqslant s$.
	Then
	\begin{align}
		C_\Gamma(-\mu/s - t, 2-k/s - t) &= (t-1)((v - 2k + \lambda)t - (2s-r)(r+1)); \label{eqn:SRG-Pol} \\
		C_\Gamma(-\mu/s - t, 1-k/s - t) &= t( (v - 2k + \lambda)(t-1) + r(r+1)). \label{eqn:SRGPol}
		% C_\Gamma(-\mu/s - t, 1-k/s - t) &= t((v - 2k + \lambda)t - (v - 2k + \lambda- r^2 - r)); \label{eqn:SRG-2Pol} \\
		% C_\Gamma(-\mu/s -1 + t, -k/s + t) &= (t-1)( (v - 2k + \lambda)t - r(r+1)). \label{eqn:SRG+Pol}
	\end{align}
	Moreover, if $\Gamma$ is co-connected then these polynomials have positive leading coefficients.
\end{lemma}

% Note that if $\Gamma$ is complete multipartite then $v - 2k + \lambda = 0$.
% v - k + sr + s + r = v - 2k + \lambda
% if \Gamma is multipartite then v - 2k + \lambda = 0
\begin{proof}
	The equalities follow from direct calculation (see Appendix~\ref{sec:comp}), using the equalities in Proposition~\ref{pro:SRG}.
	By Proposition~\ref{pro:SRGbound} if $\Gamma$ is co-connected then the polynomials have positive leading coefficients.
\end{proof}

Note that Lemma~\ref{lem:CAP0} is a special case of Lemma~\ref{lem:CAP1}, for $t=0$.
% Hence, if $\Gamma$ is co-connected, $C_\Gamma(-\mu/s - t, 2-k/s - t)$ is a positive quadratic polynomial in $t$ with zeros $1$ and $\frac{(2s+r)(r+1)}{v - 2k + \lambda} < 0$.

\begin{corollary}\label{cor:meetDelsarte1}
	Let $\Gamma$ be a type-II strongly regular graph with parameters $(v,k,\lambda,\mu)$ and eigenvalues $k > r \geqslant s$.
	Then $$
	C_\Gamma(\lfloor -\mu/s \rfloor, \lfloor 2-k/s \rfloor) < 0.
	$$
\end{corollary}
\begin{proof}
	If $\Gamma$ is disconnected then we have $\mu = 0$ and $2-k/s = \lambda + 3$.
	Whence $C_\Gamma(\lfloor -\mu/s \rfloor, \lfloor 2-k/s \rfloor) = C_\Gamma(0, \lambda+3) = -(\lambda + 3)(\lambda + 2) < 0$, as required.
	Hence we can assume that $\Gamma$ is connected.
	
	Let $t = \frc{-\mu/s}$.
	Then, using Proposition~\ref{pro:fracPart} and Equation~\eqref{eqn:SRG-Pol}, we have 
	\[
		C_\Gamma(\lfloor -\mu/s \rfloor, \lfloor 2-k/s \rfloor) = (t-1)((v - 2k + \lambda)t - (2s-r)(r+1)).
	\]
	Suppose first that $\Gamma$ is co-connected.
	The right-hand side of Equation~\eqref{eqn:SRG-Pol} is negative on the open interval $(\eta,1)$, where $\eta = (2s-r)(r+1)/(v - 2k + \lambda)$ is negative.
	Hence the corollary holds for $\Gamma$.
	On the other hand, for complete multipartite graphs we have $t = 0$, in which case the right-hand side of Equation~\eqref{eqn:SRG-Pol} is negative.
\end{proof}

The next corollary follows similarly, using the fact that the right-hand side of Equation~\eqref{eqn:SRGPol} is negative on the open interval $(0,\eta)$, where $\eta = 1-(r^2 + r)/(v - 2k + \lambda)$.

\begin{corollary}\label{cor:beatDel01}
	Let $\Gamma$ be a co-connected type-II strongly regular graph with parameters $(v,k,\lambda,\mu)$ and eigenvalues $k > r \geqslant s$.
	Suppose that 
	\[
		0< \frc{-k/s} < 1-(r^2 + r)/(v - 2k + \lambda).
	\]
	Then $C_\Gamma(\lfloor -\mu/s \rfloor, \lfloor 1-k/s \rfloor) < 0$.
\end{corollary}

\begin{remark}
	\label{rem:4.10}
	We remark that if a type-II strongly regular graph satisfies the hypothesis of Corollary~\ref{cor:beatDel01} then its complement cannot also satisfy the hypothesis.
	Indeed, suppose that $\Gamma$ satisfies the hypothesis of Corollary~\ref{cor:beatDel01}.
	Since $\frc{-k/s} > 0$ we have that $s \ne -1$ and hence $\Gamma$ is connected.
	Then, using Proposition~\ref{pro:comp}, we see that the complement of $\Gamma$ also satisfies the hypothesis of Corollary~\ref{cor:beatDel01} if
	\[
		0 < \frc{(v-k-1)/(r+1)} < 1-(s^2 +s)/\mu.
	\]
	In particular, for the corollary to hold for both $\Gamma$ and its complement, we must have both $(r^2 + r)/(v - 2k + \lambda) < 1$ and $(s^2 + s)/\mu < 1$.
	But we find that $(r^2 + r)/(v - 2k + \lambda) < 1$ if and only if $(s^2 + s)/\mu > 1$.
	One can see this by using the equalities in Proposition~\ref{pro:SRG} (see Appendix~\ref{sec:comp}) to obtain the equality 
	\[
		\mu(v-2k+\lambda) = (r^2+r)(s^2+s).
	\]
	% ***Further we note that the parameter tuple $(205, 108, 51, 63)$ satisfies the hypothesis of Corollary~\ref{cor:beatDel01}.***
\end{remark}

% subsection non_conference_strongly_regular_graph (end)

% section evaluating_the_clique_adjacency_polynomial (end)

\section{How sharp is the clique adjacency bound?} % (fold)
\label{sec:a_limitation_of_the_bound}

In this section we show that the clique adjacency bound is sharp for strongly regular graphs in certain instances.
We also comment on the sharpness of the clique adjacency bound for general strongly regular graphs.

\begin{theorem}\label{thm:limitation}
	Let $\Gamma$ be a strongly regular graph with parameters $(v,k,\lambda,\mu)$ and eigenvalues $k > r \geqslant s$.
	Suppose that $\lambda + 1 \leqslant -k/s$.
	Then the clique adjacency bound is equal to $\lambda + 2$.
\end{theorem}

As the second author observed in \cite{Soi:CAB15}, for an edge-regular graph with parameters $(v,k,\lambda)$ we have $C_\Gamma(0, y) = -y(y - 1)(y-(\lambda + 2))$, so for all $y > \lambda+2$, we have $C_\Gamma(0,y)<0$.
Hence the clique adjacency bound is always at most the trivial bound of $\lambda+2$.
Therefore, to prove Theorem~\ref{thm:limitation}, it suffices to show that, for strongly regular graph parameters satisfying $\lambda + 1 \leqslant -k/s$, the clique adjacency bound is at least $\lambda+2$.
%
% \begin{lemma}\label{lem:lam-1le}
% 	Let $\Gamma$ be a strongly regular graph with parameters $(v,k,\lambda,\mu)$.
% 	Then $C_\Gamma(-1, \lambda+2) \geqslant 0$ with equality if and only if $\overline \Gamma$ complete multipartite.
% \end{lemma}
% \begin{proof}
% 	Using Proposition~\ref{pro:SRG} we can write $C_\Gamma(-1, \lambda+2) = (\lambda+2)(k-\lambda-1)$.
% 	Then apply the inequality $k-\lambda-1 \geqslant 0$ to find that $C_\Gamma(-1, \lambda+2) \geqslant 0$ with equality if and only if $\lambda=k-1$.
% \end{proof}

\begin{lemma}\label{lem:lam1le}
	Let $\Gamma$ be a connected type-II strongly regular graph with parameters $(v,k,\lambda,\mu)$ and eigenvalues $k > r > s$.
	Suppose that $\lambda + 1 \leqslant -k/s$.
	Then $C_\Gamma(1, \lambda+2) \geqslant 0$ with equality if and only if $\lambda = -k/s - 1$.
\end{lemma}
\begin{proof}
	First suppose $\lambda + 1 = -k/s$.
	Equivalently, since $\lambda + 1 = k + (r+1)(s+1)$ and $\mu = k + rs$, we have $-\mu/s = 1$.
	In this case, $C_\Gamma(1, \lambda+2) = C_\Gamma(-\mu/s, 1-k/s)$, which is zero by Equation~\eqref{eqn:SRGPol}.
	
	It remains to assume $\lambda+1 < -k/s$.
	Using Proposition~\ref{pro:SRG} (see Appendix~\ref{sec:comp}) we can write
	\[
		\frac{\mu}{2} C_\Gamma(1,\lambda+2) = k(k-(\mu+1)(\lambda+1))+\mu(\lambda+1)^2.
	\]
	To show this quantity is nonnegative, it suffices to show that $k-(\mu+1)(\lambda+1)$ is nonnegative.
	Using the inequality $\lambda+1 < -k/s$, we have $k-(\mu+1)(\lambda+1) > k(1+(\mu+1)/s)$.
	It therefore suffices to show that $1+(\mu+1)/s \geqslant 0$.
	
	Since $\lambda = k + r + s + rs$, the inequality $\lambda+1 < -k/s$ becomes 
	\[
		k + (r+1)(s+1) < -k/s.
	\]
	Since $s < -1$, it follows that $r+1 > -k/s$.
	Multiplying this inequality by $-s$ gives $-s(r+1) > k$.
	Since both $s$ and $r$ are integers, we have $-s(r+1) \geqslant k + 1$
	Now by rearranging and substituting $\mu = k+rs$, we obtain the inequality $1+(\mu+1)/s \geqslant 0$ as required.
\end{proof}
%
% In \cite{Soi:CAB15}, Soicher observed the following.
%
% \begin{lemma}\label{lem:SoicherObs}
% 	Let $\Gamma$ be a type-II $(v,k,\lambda,\mu)$-strongly regular graph.
% 	Then $C_\Gamma(0,y) = y(y-1)(\lambda+2-y)$.
% \end{lemma}

\begin{lemma}\label{lem:NoOtherZeros}
	Let $\Gamma$ be an edge-regular graph with parameters $(v,k,\lambda)$ 
	such that $C_\Gamma(b,\lambda + 2)\geqslant 0$ for all integers $b$. 
	Then $C_\Gamma(b,c)\geqslant 0$ for all $c\in \{2,\ldots,\lambda+2\}$ and all integers $b$.
\end{lemma} 

\begin{proof}
	Let $c\in \{2,\ldots,\lambda+2\}$ and let $b$ be an integer. 
	If $b\leqslant 0$, then from the definition of the clique adjacency polynomial $C_\Gamma(x,y)$,
	we see that $C_\Gamma(b,c)\geqslant 0$, so we now assume that $b$ is positive.

	A calculation (see Appendix~\ref{sec:comp}) shows that
	\[ 
		C_\Gamma(b,c)-C_\Gamma(b,\lambda+2) = (\lambda+2-c)(b-c)(b-c+1) + 2b(\lambda+2-c)(k-\lambda-1).
	\]
	This quantity is nonnegative since $b$ and $\lambda+2-c$ are nonnegative integers,
	the product of two consecutive integers is nonnegative, and $k-\lambda-1$ is also nonnegative by Proposition~\ref{pro:SRGbound}.
	Hence 
	\[ 
		C_\Gamma(b,c)\geqslant C_\Gamma(b,\lambda+2)\geqslant 0,
	\]
	as required.
\end{proof}

Now we prove Theorem~\ref{thm:limitation}.

\begin{proof}[Proof of Theorem~\ref{thm:limitation}]
	Firstly, if $\Gamma$ is disconnected then $\Gamma$ is the disjoint union of complete graphs and hence contains cliques of size $\lambda + 2$.
	Therefore the clique adjacency bound is at least $\lambda + 2$.
	Now we assume that $\Gamma$ is connected.
	
	By Lemma~\ref{lem:NoOtherZeros}, the clique adjacency bound is less than $\lambda+2$ only if there exists some integer $b$ such that $C_\Gamma(b,\lambda + 2)$ is less than zero.
	To ease notation set $f(x) := C_\Gamma(x,\lambda + 2)$.
	Hence 
	\[
		f(x) = x \left ( (v-\lambda-2)x + v + (2\lambda-2k+1)(\lambda+2) \right ). 
	\]
	It suffices to show that there does not exist any integer $b$ such that $f(b) < 0$.
	
	Observe that the polynomial $f(x)$ is a quadratic polynomial in the variable $x$.
	Furthermore, the leading coefficient of $f(x)$ is $v-\lambda-2 \geqslant 0$, and $f(0) = 0$.
	Let $\xi$ be the other zero of $f(x)$.
	Now, $f(x)$ is negative if and only if $x$ is between $0$ and $\xi$.
	Hence, if $f(-1)$ and $f(1)$ are both nonnegative then there are no integers $b$ such that $f(b) < 0$.
	As in the proof of the previous result $f(-1)$ is nonnegative.
	Therefore Lemma~\ref{lem:lam1le} completes the proof for type-II strongly regular graphs.
	
	The inequality $\lambda+1 \leqslant -k/s$ only holds for type-I strongly regular graphs on $5$ vertices or $9$ vertices (where we have equality).
	One can explicitly compute the clique adjacency bound for these two cases: the unique $(5,2,0,1)$-strongly regular graph and the unique $(9,4,1,2)$-strongly regular graph.
	For each of these graphs the clique adjacency bound is equal to $\lambda+2$.
\end{proof}

Now we give a couple of remarks about Theorem~\ref{thm:limitation}.

\begin{remark}
	For strongly regular graphs with $\lambda \leqslant 1$, it is easy to see that the clique number is $\lambda + 2$.
	By Theorem~\ref{thm:limitation}, the clique adjacency bound is equal to the clique number for such graphs.
	Let $\Gamma$ be a strongly regular graph with parameters $(v,k,\lambda,\mu)$.
	By Proposition~\ref{pro:SRG}, we see that $k = -s(r+1)- r +\lambda$.
	Therefore, for strongly regular graphs with $\lambda=2$ and $r \geqslant 2$, we have $\lambda+1 = 3 \leqslant -k/s$, and so Theorem~\ref{thm:limitation} applies to such graphs.
	% The authors do not know of any type-II parameters $(v,k,2,\mu)$ all of whose strongly regular graphs have clique number less than $4$.
	% *** $K_{2,2,2}$ has parameters $(6,4,2,4)$. $(r = 0).$***
\end{remark}

\begin{remark}
	We conjecture that if the clique adjacency bound is less than $-k/s$ then $\lambda+1 \leqslant -k/s$.
	We have verified this conjecture for all feasible parameter tuples for strongly regular graphs on up to $1300$ vertices, making use of Brouwer's website~\cite{Brouwer:URL}.
\end{remark}

In Table~\ref{tab:bd}, we list all the feasible parameter tuples for strongly regular graphs on at most $150$ vertices to which we can apply either Theorem~\ref{thm:beatDelsarteConf} or Theorem~\ref{thm:beatDelsarteNonConf}.
In other words, Table~\ref{tab:bd} displays the feasible parameters for strongly regular graphs on at most $150$ vertices for which the clique adjacency bound is strictly less than the Delsarte bound.
In the column labelled `Exists', if there exists a strongly regular graph with the appropriate parameters then we put `+', or `!' if the graph is known to be unique; otherwise, if the existence is unknown, we put `?'.  
In the final column of Table~\ref{tab:bd}, we put `Y' (resp. `N') if there exists (resp. does not exist) a strongly regular graph with the corresponding parameters that has clique number equal to the clique adjacency bound, otherwise we put a `?' if such existence is unknown.
We refer to Brouwer's website~\cite{Brouwer:URL} for details on the existence of strongly regular graphs with given parameters.

For the parameter tuples in Table~\ref{tab:bd}, the Delsarte bound is equal to the clique adjacency bound plus $1$.
As an example of a parameter tuple for which the clique adjacency bound differs from the Delsarte bound by $2$, we have $(378,52,1,8)$ for which there exists a corresponding graph~\cite{Penttila05}.
For this graph the Delsarte bound is $5$, but the clique adjacency bound is $3$.

	\begin{table}[htbp]
		\begin{center}
			\begin{tabular}{l | c | c | c | c  }
				Parameters & Type & CAB & Exists & Sharp  \\
				\hline
				$(17, 8, 3, 4)$      &  I   & $3$  & !    	& Y    \\
			  	$(37, 18, 8, 9)$     &  I   & $5$  & +  	& Y 	\\
				$(50, 7, 0, 1)$      &  II  & $2$  & ! 		& Y    \\ 
				$(56, 10, 0, 2)$     &  II  & $2$  & ! 		& Y 	\\	  		
				$(65, 32, 15, 16)$   &  I   & $7$  & ?    	& ?    \\
				$(77, 16, 0, 4)$     &  II  & $2$  & !   	& Y    \\
				$(88, 27, 6, 9)$     &  II  & $4$  & ?    	& ?    \\
				$(99, 14, 1, 2)$     &  II  & $3$  & ?    	& Y    \\
				$(100, 22, 0, 6)$    &  II  & $2$  & !    	& Y    \\
				$(101, 50, 24, 25)$  &  I   & $9$  & +    	& ?    \\
				$(105, 32, 4, 12)$   &  II  & $3$  & !    	& Y    \\
				$(111, 30, 5, 9)$    &  II  & $4$  & ?    	& ? 	\\
				$(115, 18, 1, 3)$    &  II  & $3$  & ?    	& Y 	\\
				$(120, 42, 8, 18)$   &  II  & $3$  & !    	& Y 	\\
				$(121, 36, 7, 12)$   &  II  & $4$  & ?    	& ? 	\\
				$(133, 32, 6, 8)$    &  II  & $5$  & ?    	& ? 	\\
				$(144, 39, 6, 12)$   &  II  & $4$  &  +    	& Y 	\\
				$(144, 52, 16, 20)$  &  II   & $6$  &  ?    & ? 	\\
				$(145, 72, 35, 36)$  &  I   & $11$  &  ?    & ? 	\\
				$(149, 74, 36, 37)$  &  I   & $11$  &  +    & ? 
			\end{tabular}
		 \end{center}
	 	\caption{Feasible parameter tuples for strongly regular graphs on at most $150$ vertices to which we can apply either Theorem~\ref{thm:beatDelsarteConf} or Theorem~\ref{thm:beatDelsarteNonConf}.}
	 	\label{tab:bd}
	\end{table}
	
	Feasible parameters for which there does not exist a corresponding strongly regular graph whose clique number is equal to the clique adjacency bound include $(16,10,6,6)$ and $(27, 16, 10, 8)$.
	However, we ask the following question.
	Do there exist strongly regular graphs with parameters $(v,k,\lambda,\mu)$, with $k < v/2$, such that every strongly regular graph having those parameters has clique number less than the clique adjacency bound?
	
	\section{Hoffman bound vs Delsarte bound vs clique adjacency bound}
\label{sec:HoffmanDelsarteCAB}
	Let $\Gamma$ be a connected non-complete regular graph with $v$ vertices, valency $k$, and second
	largest eigenvalue $r<k$. Then the complement $\overline\Gamma$ of $\Gamma$ is a regular graph with valency 
	$\overline k=v-k-1$ and least eigenvalue $\overline s=-r-1<0$. We may obtain a bound for the clique number
	of $\Gamma$ by applying the 
	Hoffman bound (also called the ratio bound) \cite[Theorem~2.4.1]{GM16} on the
	size of a largest independent set (coclique) of $\overline\Gamma$. This gives
	% \[ \omega(\Gamma)\leqslant \lfloor v/(1-\overline{k}/\overline{s})\rfloor.\]
	\[ 
	\omega(\Gamma)\leqslant \left \lfloor \frac{v}{1-\overline{k}/\overline{s}} \right \rfloor.
	\]
	If $\Gamma$ is strongly regular, then it is known
	(and follows from the relations of Proposition~\ref{pro:SRG})
	that the Delsarte bound for $\omega(\Gamma)$ is the same as  
	that given by the Hoffman bound above.

	Now the Delsarte bound applies not only to strongly regular
	graphs, but also to the graphs $\{\Gamma_1,\ldots,\Gamma_d\}$ 
	of the relations (other than equality) of any   
	$d$-class symmetric association scheme (see \cite[Corollary~3.7.2]{GM16}). Thus,
	if $\Gamma$ is such a graph, having valency $k$ and 
	least  eigenvalue $s$, then $\omega(\Gamma)\leqslant \lfloor 1-k/s\rfloor$.

	Here is an interesting illustrative example.  Let $\Delta$ be the edge graph (or line
	graph) of the incidence graph of the projective plane of order $2$.  Then $\Delta$
	is the unique distance-regular graph with intersection array 
	$\{4,2,2;1,1,2\}$. Now let $\Delta_3$ be
	the graph on the vertices of $\Delta$, with two vertices joined 
	by an edge if and only if they have distance~$3$ in $\Delta$. 
	Then $\Delta_3$ is the graph of a relation in the usual 
	symmetric association scheme associated
	with a distance-regular graph, where two vertices are in relation~$i$
	precisely when they are at distance~$i$ in the distance-regular graph.   
	The graph $\Delta_3$ has diameter~$2$ and is edge-regular (but not strongly
	regular) with parameters $(v,k,\lambda)=(21,8,3)$.
	The clique adjacency bound for $\Delta_3$ is $4$.  
	The least eigenvalue of $\Delta_3$ is $-\sqrt{8}$,
	and the Delsarte bound gives $3$, and indeed, 
	$\omega(\Delta_3)=3$.  However, the complement of $\Delta_3$ has least
	eigenvalue $-1-\sqrt{8}$, and the Hoffman bound for independent sets in the
	complement of $\Delta_3$ gives $5$. Thus, for $\Delta_3$, the Delsarte bound is better
	than the clique adjacency bound which is better than that obtained from the Hoffman bound.
	However, the three bounds are for different classes of graphs.
	For example, there may well be an edge-regular graph with parameters
	$(21,8,3)$ and clique number $4$. It would be interesting to find one.

	We conjecture that if $\Gamma$ is any connected non-complete edge-regular graph, 
	then the clique adjacency bound for $\omega(\Gamma)$ is at most that 
	obtained from the Hoffman bound for $\overline \Gamma$.   
\appendix
	\section{Algebraic computational verification of identities}
\label{sec:comp}
	In this appendix we present the algebraic computations
	in Maple~\cite{maple} that were used to verify
	certain identities employed in this paper. These identities
	were also checked independently using Magma~\cite{Magma}.    

	We start up Maple (version 18) and assign to $C$ the clique adjacency polynomial.
	\begin{verbatim}
	> C:=x*(x+1)*(v-y)-2*x*y*(k-y+1)+y*(y-1)*(lambda-y+2):
	\end{verbatim}
	We then make a set of relators, obtained from Proposition~\ref{pro:SRG},
	which evaluate to $0$ on the parameters $(v,k,\lambda,\mu)$ and eigenvalues 
	$r,s$ (with $k>r \geqslant s$) of a strongly regular graph. 
	\begin{verbatim}
	> srg_rels:={(v-k-1)*mu-k*(k-lambda-1),(lambda-mu)-(r+s),(mu-k)-r*s}:
	\end{verbatim}
	We make a further set of relators 
	which evaluate to $0$ on the parameters and eigenvalues 
	of a type-I strongly regular graph. 
	\begin{verbatim}
	> type1_rels:=srg_rels union {2*k-(v-1),4*lambda-(v-5),4*mu-(v-1)}:
	\end{verbatim}

	Let $R=\mathbb{Q}[t,v,k,\lambda,\mu,r,s]$ be the ring 
	of polynomials over $\mathbb{Q}$ in the indeterminates
	$t,v,k,\lambda,\mu,r,s$, let $S$ be the ideal of $R$
	generated by \texttt{srg\_rels} given above, and let 
	$T$ be the ideal of $R$ generated by \texttt{type1\_rels}. 
	We use the Maple package \texttt{Groebner}
	to caclulate and employ Gr\"{o}bner bases \cite{Co99} 
	to work in the factor rings $R/S$ and $R/T$.

	We set the monomial ordering for the 
	Gr\"{o}bner basis calculations to be the
	Maple \texttt{tdeg} ordering, more commonly called
	the grevlex ordering, with the indeterminates ordered as
	 $t>v>k>\lambda>\mu>r>s$.
	\begin{verbatim} 
	> ordering:=tdeg(t,v,k,lambda,mu,r,s):
	\end{verbatim}
	 Then we compute a Gr\"{o}bner basis $G$ for $S$.
	\begin{verbatim}
	> G:=Groebner[Basis](srg_rels,ordering):
	\end{verbatim}
	For the record, $G=
	[\lambda-\mu-r-s,rs+k-\mu,{k}^{2}-kr-ks-\mu\,v-k+\mu]$. 
	Similarly, we compute a Gr\"{o}bner basis $H$ for $T$.
	\begin{verbatim} 
	> H:=Groebner[Basis](type1_rels,ordering):
	\end{verbatim}
	Here, we obtain $H=
	[r+1+s,\lambda+1-\mu,k-2\,\mu,v-1-4\,\mu,{s}^{2}-\mu+s]$. 

	We now verify that the identity of Lemma~\ref{lem:CAP0} holds, by checking that
	$s^3(C(-\mu/s,2-k/s)-(2s-r)(r+1))=0$ in $R/S$. 
	\begin{verbatim}
	> Groebner[NormalForm](expand(s^3*(eval(C,[x=-mu/s,y=2-k/s])
	>    - (2*s-r)*(r+1))),G,ordering);
	                                       0
	\end{verbatim}
	Similarly, we verify that the identities of Lemma~\ref{lem:CAP2.0} hold for type-I
	strongly regular graphs, by working in $R/T$.
	\begin{verbatim}
	> Groebner[NormalForm](eval(C,[x=r-t,y=3+2*r-2*t])
	>    - 2*(t-1)*(t+s-2)*(t+2*s),H,ordering);
	                                       0
	> Groebner[NormalForm](eval(C,[x=r-t,y=2+2*r-2*t])
	>    - (t+s)*(2*t^2+(4*s-1)*t-3*s-1),H,ordering);
	                                       0
	\end{verbatim}
	Next are the verifications of the identities of Lemma~\ref{lem:CAP1}.
	\begin{verbatim}
	> Groebner[NormalForm](expand(s^3*(eval(C,[x=-mu/s-t,y=2-k/s-t])
	>    - (t-1)*((v-2*k+lambda)*t-(2*s-r)*(r+1)))),G,ordering); 
	                                       0
	> Groebner[NormalForm](expand(s^3*(eval(C,[x=-mu/s-t,y=1-k/s-t])
	>    - t*((v-2*k+lambda)*(t-1)+r*(r+1)))),G,ordering); 
	                                       0
	\end{verbatim}
	Here is confirmation of the identity used in Remark~\ref{rem:4.10}.
	\begin{verbatim}
	> Groebner[NormalForm](mu*(v-2*k+lambda)-(r^2+r)*(s^2+s),G,ordering); 
	                                       0
	\end{verbatim}
	Next is verification of the identity used in the proof of Lemma~\ref{lem:lam1le}.
	\begin{verbatim}
	Groebner[NormalForm](mu*eval(C,[x=1,y=lambda+2])/2 
	  - (k*(k-(mu+1)*(lambda+1))+mu*(lambda+1)^2),G,ordering); 
	                                      0
	\end{verbatim}
	Finally, here is a confirmation of the polynomial equality used in the proof
	of Lemma~\ref{lem:NoOtherZeros}.
	\begin{verbatim}
	> expand((eval(C,[x=b,y=c])-eval(C,[x=b,y=lambda+2]))
	>    - ((lambda+2-c)*(b-c)*(b-c+1)+2*b*(lambda+2-c)*(k-lambda-1)));
	                                       0
	\end{verbatim}

	We remark that the total CPU time for all these computations on
	a desktop Linux PC was only about 0.16 seconds, and the total
	store used by Maple was 2.4MB.
	
	\section*{Acknowledgement} % (fold)
	
	We thank Anton Betten for organising the Combinatorics and Computer Algebra 2015 conference, whose problem sessions brought us together to begin this work.

\bibliographystyle{myplain}
\bibliography{sbib}

\end{document}